\documentclass[11pt]{article}

\usepackage[T1]{fontenc}
\usepackage[utf8]{inputenc}
\usepackage{lmodern}
\usepackage{microtype}
\usepackage[a4paper,margin=28mm]{geometry}
\usepackage{amsmath,amssymb,amsthm,mathtools,mathrsfs,needspace,array}
\usepackage{aliascnt}
\usepackage{xcolor}
\usepackage{hyperref}
\usepackage[nameinlink,capitalise,noabbrev]{cleveref}
\usepackage{fancyhdr}
\hypersetup{
  colorlinks=true,linkcolor=blue!45!black,citecolor=green!35!black,
  urlcolor=blue!55!black,
  pdftitle={Aspherical Symplectic Manifolds with Positive First Chern Degree},
  pdfauthor={Jingcao Wu},
  pdfsubject={A counterexample to the symplectic uniruledness question of BDPP},
  pdfkeywords={symplectic manifold, first Chern class, aspherical manifold, torus bundle, arithmetic Fuchsian group}
}
\newtheorem{theorem}{Theorem}[section]
\newaliascnt{proposition}{theorem}
\newtheorem{proposition}[proposition]{Proposition}
\aliascntresetthe{proposition}
\newaliascnt{lemma}{theorem}
\newtheorem{lemma}[lemma]{Lemma}
\aliascntresetthe{lemma}
\theoremstyle{remark}
\newaliascnt{remark}{theorem}
\newtheorem{remark}[remark]{Remark}
\aliascntresetthe{remark}

\newcommand{\C}{\mathbb C}
\newcommand{\Q}{\mathbb Q}
\newcommand{\R}{\mathbb R}
\newcommand{\Z}{\mathbb Z}
\newcommand{\HH}{\mathbb H}
\newcommand{\cO}{\mathcal O}
\newcommand{\ii}{\mathrm i}
\newcommand{\qi}{\mathbf i}
\newcommand{\qj}{\mathbf j}
\newcommand{\qk}{\mathbf k}
\DeclareMathOperator{\Trd}{Trd}
\DeclareMathOperator{\Nrd}{Nrd}
\DeclareMathOperator{\SL}{SL}
\DeclareMathOperator{\Id}{Id}
\numberwithin{equation}{section}

\title{\bfseries Aspherical Symplectic Manifolds\\with Positive First Chern Degree}
\author{Jingcao Wu\thanks{School of Mathematics, Shanghai University of Finance and Economics, Shanghai 200433, People's Republic of China. E-mail: \texttt{wujincao@shufe.edu.cn}.}}
\date{}

\begin{document}
\maketitle

\begin{abstract}
For every integer $n\ge5$, we construct a closed aspherical symplectic manifold of real dimension $2n$ with positive first Chern degree. Each example is a flat torus bundle over the same compact hyperbolic surface, with arithmetic monodromy arising from an explicit indefinite quaternion division algebra over $\Q$. We identify its symplectic first Chern class with the pullback of the canonical class of the base and compute its degree explicitly. Asphericity rules out nonconstant pseudoholomorphic spheres for every tame almost complex structure. Thus these examples give a negative answer to the symplectic uniruledness question of Boucksom--Demailly--P\u{a}un--Peternell in every even real dimension at least ten.
\end{abstract}

\medskip
\noindent\textbf{2020 Mathematics Subject Classification.}
Primary 53D35; Secondary 53D05, 57R20.

\noindent\textbf{Keywords.}
Symplectic uniruledness; first Chern class; aspherical manifolds.

\section{Introduction}

Boucksom--Demailly--P\u{a}un--Peternell (BDPP) asked whether positivity of the first Chern degree forces a compact symplectic manifold to be covered by rational pseudoholomorphic curves \cite[Question~5.3]{BDPP13}. More precisely, given a closed symplectic manifold $(M^{2n},\omega)$ and an $\omega$-compatible almost complex structure $J$, does
\begin{equation}\label{eq:bdpp-condition}
 \bigl\langle c_1(TM,J)\smile[\omega]^{n-1},[M]\bigr\rangle>0
\end{equation}
imply that $M$ is covered by images of nonconstant $J$-holomorphic maps from $\mathbb P^1$? The question asks whether the relationship between canonical positivity and uniruledness in complex geometry extends to the symplectic setting; it does not assume that $J$ is integrable.

We answer this question negatively in every even real dimension at least ten. Our examples satisfy the numerical hypothesis, yet their topology excludes all nonconstant pseudoholomorphic spheres, regardless of the choice of tame almost complex structure.

We write $c_1(M,\omega)$ for the first Chern class of an $\omega$-compatible almost complex structure. This class is independent of the choice of compatible structure, since the space of such structures is contractible. We call the pairing in \eqref{eq:bdpp-condition} the \emph{first Chern degree}; unlike a Chern number, it depends on the symplectic class $[\omega]$.

\begin{theorem}\label{thm:main}
There exists a compact Riemann surface $\Sigma$ of genus $g\ge2$ such that, for every integer $n\ge5$, there exist a closed connected symplectic manifold $(X_n^{2n},\omega_n)$ and a smooth fiber bundle
\[
 T^{2n-2}\longrightarrow X_n\xrightarrow{p_n}\Sigma
\]
with the following properties:
\begin{enumerate}
 \item the universal cover of $X_n$ is diffeomorphic to $\HH\times\R^{2n-2}$, and hence $X_n$ is aspherical;
 \item the first Chern class satisfies
 \[
  c_1(X_n,\omega_n)=p_n^*c_1(K_\Sigma)\quad\text{in }H^2(X_n;\Z);
 \]
 \item with the normalization of the fiber form specified below,
 \[
  \bigl\langle c_1(X_n,\omega_n)\smile[\omega_n]^{n-1},[X_n]\bigr\rangle=9(n-1)!(2g-2)>0.
 \]
\end{enumerate}
For every $\omega_n$-tame almost complex structure $J'$, every $J'$-holomorphic map $\mathbb P^1\to X_n$ is constant.
\end{theorem}

Thus BDPP's proposed conclusion fails even in the weaker form asserting the existence of a single nonconstant pseudoholomorphic sphere. In real dimension ten, the first Chern degree is $216(2g-2)$. Taking products with standard symplectic tori gives all higher even dimensions while preserving the first Chern class pulled back from the base.

\noindent\textbf{The construction.}
The arithmetic ingredients come from the theory of compact Shimura curves associated with indefinite quaternion division algebras; see \cite[Example~0.4 and \S5]{VZ04}, where families of false elliptic curves, namely abelian surfaces with quaternionic multiplication, are discussed. We use the underlying arithmetic representation to construct the integral monodromy of a flat torus bundle. Specifically, let
\[
B=\left(\frac{-1,3}{\Q}\right),\qquad\cO=\Z\langle1,\qi,\qj,\qi\qj\rangle,
\]
and choose an explicit torsion-free congruence subgroup $\Gamma$ of $\cO^1$. Left multiplication preserves an integral alternating form on $B\otimes_\Q\R$. Two copies of this representation, together with a trivial symplectic representation of real dimension $2n-10$, provide the monodromy of the required $T^{2n-2}$-bundle.

The key geometric step is to equip the flat real symplectic bundle with a compatible complex structure and determine the sign of its first Chern class. Let $F$ be the flat real plane bundle over $\Sigma=\Gamma\backslash\HH$ associated with the standard representation. We construct a compatible complex structure $I$ for which
\[
 (F,I)^{\otimes2}\simeq K_\Sigma.
\]
Each quaternionic summand contributes two copies of $(F,I)$, with the same symplectic sign. The complex bundle over $\Sigma$ associated with the vertical monodromy is $(F,I)^{\oplus4}\oplus\underline{\C}^{\,n-5}$, where $\underline{\C}^{\,n-5}$ is trivial. Its first Chern class is $2c_1(K_\Sigma)$, which combines with the contribution $-c_1(K_\Sigma)$ from $T\Sigma$.

Flat real plane bundles with nonzero Euler class are classical \cite{Milnor58}. Here the compatible complex structure is not parallel, so flatness of the real bundle does not force its first Chern class to vanish. The invariant integral lattice yields compact torus fibers, while asphericity and positivity of symplectic area exclude pseudoholomorphic spheres \cite{MS12}.

Hyperbolic geometry also features in the constructions of Fine--Panov \cite{FP10}, which produce non-K\"ahler symplectic Calabi--Yau manifolds and higher-dimensional symplectic Fano manifolds. For the natural almost complex structure on a hyperbolic twistor space, \cite[\S6.3, Lemma~42]{FP10} proves that every pseudoholomorphic sphere lies in a fiber. This prevents such curves from connecting points in distinct fibers, but does not rule out rational curves through individual points. By contrast, the manifolds in \cref{thm:main} admit no nonconstant pseudoholomorphic spheres for any tame almost complex structure. Their positivity is precisely the numerical condition \eqref{eq:bdpp-condition}: they are not symplectic Fano, since $c_1(X_n,\omega_n)^2=0$; see \cref{sec:consequences}.

More generally, using $r$ quaternionic summands gives symplectic manifolds $M_r^{4r+2}$ satisfying
\[
 c_1(M_r,\omega_{r,t})=(r-1)q_r^*c_1(K_\Sigma).
\]
For $r=1$, the first Chern class vanishes; for $r=2$, we obtain the ten-dimensional factor of $X_n=M_2\times T^{2n-10}$ used to prove \cref{thm:main}.

The paper is organized as follows. \Cref{sec:arithmetic} constructs the integral symplectic monodromy. \Cref{sec:plane} determines the compatible first Chern class, including its sign, through an automorphy-factor calculation. \Cref{sec:bundle} constructs the torus bundles and proves \cref{thm:main}. Finally, \cref{sec:consequences} studies the displayed almost complex structures, records the characteristic classes, and discusses the dimensions covered by the construction.

\section{An integral symplectic representation}\label{sec:arithmetic}

\subsection{The quaternion algebra and the base curve}

Let $B$ be the quaternion algebra over $\Q$ generated by $\qi$ and $\qj$ with
\[
 \qi^2=-1,\qquad \qj^2=3,\qquad \qi\qj=-\qj\qi,
\]
and put $\qk=\qi\qj$. Quaternionic conjugation sends $\qi,\qj,\qk$ to their negatives. The reduced trace and norm are
\[
 \Trd(x)=x+\bar x,\qquad \Nrd(x)=x\bar x.
\]
In coordinates,
\[
 \Nrd(a+b\qi+c\qj+d\qk)=a^2+b^2-3c^2-3d^2.
\]

\begin{lemma}
The algebra $B$ is a division algebra over $\Q$, and $B_\R:=B\otimes_\Q\R$ is isomorphic to $M_2(\R)$.
\end{lemma}

\begin{proof}
If a nonzero element of $B$ had zero reduced norm, clearing denominators and dividing by a common factor would give a primitive integral solution of
\[
 a^2+b^2=3(c^2+d^2).
\]
Reducing modulo $3$ gives $3\mid a$ and $3\mid b$. Substituting this back into the equation then gives $3\mid c$ and $3\mid d$, contradicting primitivity. Hence every nonzero element $x\in B$ has inverse $\bar x/\Nrd(x)$.

An explicit real splitting is given by
\begin{equation}\label{eq:splitting}
\qi\longmapsto\begin{pmatrix}0&-1\\1&0\end{pmatrix},\qquad\qj\longmapsto\begin{pmatrix}\sqrt3&0\\0&-\sqrt3\end{pmatrix}.
\end{equation}
These matrices satisfy the defining relations, and the images of $1,\qi,\qj,\qk$ form a real basis of $M_2(\R)$. Under this identification, reduced trace and norm correspond to matrix trace and determinant.
\end{proof}

The lattice $\cO=\Z\oplus\Z\qi\oplus\Z\qj\oplus\Z\qk$ is a subring of $B$ containing $1$, and hence an order. It is preserved by conjugation. Its group of norm-one units is
\[
 \cO^1=\{q\in\cO:\Nrd(q)=1\}.
\]
We view this group as a subgroup of $\SL_2(\R)$ by \eqref{eq:splitting}. Fix the principal congruence subgroup
\begin{equation}\label{eq:gamma}
 \Gamma=\{q\in\cO^1:q-1\in5\cO\}.
\end{equation}

\begin{lemma}
The subgroup $\Gamma$ is a torsion-free cocompact lattice in $\SL_2(\R)$. Its action on $\HH$ by fractional linear transformations is free, and
\[
 \Sigma=\Gamma\backslash\HH
\]
is a compact Riemann surface of genus $g\ge2$.
\end{lemma}

\begin{proof}
By the arithmetic lattice theorem for indefinite quaternion division algebras, the norm-one group of an order is a discrete cocompact subgroup of $\SL_2(\R)$; see \cite[\S5]{VZ04}. The subgroup $\Gamma$ is the kernel of the reduction homomorphism $\cO^1\to(\cO/5\cO)^\times$, whose target is finite. Hence $\Gamma$ is also a cocompact lattice.

To check torsion-freeness directly, suppose that $q\in\Gamma$ has finite order. Its matrix eigenvalues are roots of unity, and therefore
\[
 \Trd(q)\in\Z\cap[-2,2].
\]
The congruence $q\equiv1\pmod{5\cO}$ gives $\Trd(q)\equiv2\pmod5$. It follows that $\Trd(q)=2$. A finite-order matrix over a field of characteristic zero is diagonalizable, so both eigenvalues being $1$ implies $q=1$.

The stabilizer of any point of $\HH$ in $\SL_2(\R)$ is compact, and its intersection with a discrete torsion-free subgroup is trivial. Thus $\Gamma$ acts freely on $\HH$. The compact quotient inherits the hyperbolic metric, and Gauss--Bonnet yields $g\ge2$.
\end{proof}

\subsection{The invariant alternating form}

Define a bilinear form on $B_\R$ by
\begin{equation}\label{eq:quaternion-form}
 \Omega_B(x,y)=-\frac12\Trd(\qi\bar x y).
\end{equation}

\begin{lemma}\label{lem:form}
The form $\Omega_B$ is alternating, nondegenerate, and integral on $\cO$. Left multiplication by $\cO^1$ preserves both $\cO$ and $\Omega_B$. Under the identification \eqref{eq:splitting}, let $x_1,x_2$ and $y_1,y_2$ denote the respective columns of $x$ and $y$. Then
\begin{equation}\label{eq:columns}
 \Omega_B(x,y)=\frac12\bigl(\det(x_1,y_1)+\det(x_2,y_2)\bigr).
\end{equation}
In particular, $B_\R$ splits as two symplectic planes, each with one-half of the standard area form.
\end{lemma}

\begin{proof}
In the ordered basis $(1,\qi,\qj,\qk)$, direct evaluation of \eqref{eq:quaternion-form} gives
\begin{equation}\label{eq:Q}
 Q=\begin{pmatrix}
 0&1&0&0\\
 -1&0&0&0\\
 0&0&0&3\\
 0&0&-3&0
 \end{pmatrix}.
\end{equation}
Thus $\Omega_B$ is alternating, integral on $\cO$, and nondegenerate. If $q\in\cO^1$, then $q^{-1}=\bar q\in\cO$, so left multiplication is a lattice automorphism. Moreover,
\[
\Omega_B(qx,qy)=-\frac12\Trd(\qi\bar x\bar q qy)=-\frac12\Trd(\qi\bar x y)=\Omega_B(x,y).
\]
Substituting the matrices in \eqref{eq:splitting} into the right-hand side of \eqref{eq:columns} gives the same matrix \eqref{eq:Q}, proving the formula. Left multiplication acts by the standard representation on each column, and both summands occur with the positive coefficient $1/2$.
\end{proof}

For an integer $r\ge1$, set
\begin{equation}\label{eq:torus-data}
V_r=B_\R^{\oplus r},\qquad\Lambda_r=\cO^{\oplus r},\qquad\Omega_r=\Omega_B\oplus\cdots\oplus\Omega_B.
\end{equation}
The diagonal left action of $\Gamma$ preserves $\Lambda_r$ and $\Omega_r$. Consequently, it acts by symplectomorphisms on the compact torus
\[
 T_r=V_r/\Lambda_r.
\]
Unimodularity is not required: the integral, nondegenerate form \eqref{eq:Q} already defines a symplectic form on the torus.

\section{A compatible complex structure and its Chern class}\label{sec:plane}

Let
\[
 F=(\HH\times\R^2)/\Gamma\longrightarrow\Sigma,
\]
with $\gamma\cdot(z,v)=(\gamma z,\gamma v)$, and equip $\R^2$ with the standard area form $\varepsilon=dx\wedge dy$. We construct a compatible complex structure on $F$ and determine its first Chern class, paying particular attention to the sign.

\begin{lemma}
For $z=u+\ii v\in\HH$, define
\begin{equation}\label{eq:Iz}
 I_z=\frac1v
 \begin{pmatrix}
 u&-(u^2+v^2)\\
 1&-u
 \end{pmatrix}.
\end{equation}
Then $I_z$ is $\varepsilon$-compatible and satisfies
\begin{equation}\label{eq:equivariance}
I_{\gamma z}=\gamma I_z\gamma^{-1}\qquad(\gamma\in\SL_2(\R)).
\end{equation}
Thus $I$ descends to a complex structure on $F$, and the resulting complex line bundle satisfies
\begin{equation}\label{eq:half-canonical}
 (F,I)^{\otimes2}\simeq K_\Sigma.
\end{equation}
In particular,
\[
2c_1(F,I)=c_1(K_\Sigma),\qquad\langle c_1(F,I),[\Sigma]\rangle=g-1.
\]
\end{lemma}

\begin{proof}
The matrix \eqref{eq:Iz} has square $-\Id$. The symmetric matrix of the bilinear form $\varepsilon(\,\cdot\,,I_z\,\cdot\,)$ is
\begin{equation}\label{eq:compatible-metric}
 \frac1v
 \begin{pmatrix}
 1&-u\\
 -u&u^2+v^2
 \end{pmatrix}.
\end{equation}
Its upper-left entry is positive and its determinant equals $1$; hence it is positive definite.

The $+\ii$-eigenline of $I_z$ in $\C^2$ is spanned by
\[
 s(z)=\begin{pmatrix}z\\1\end{pmatrix}.
\]
If $\gamma=\left(\begin{smallmatrix}a&b\\c&d\end{smallmatrix}\right)\in\SL_2(\R)$, then
\begin{equation}\label{eq:automorphy}
 \gamma s(z)=(cz+d)s(\gamma z).
\end{equation}
Applying the conjugate identity on the $-\ii$-eigenline proves \eqref{eq:equivariance}.

The map
\[
 (\R^2,I_z)\longrightarrow\ker(I_z-\ii\Id),\qquad
 w\longmapsto w-\ii I_z w,
\]
is complex linear and $\Gamma$-equivariant. Hence $(F,I)$ identifies with the holomorphic eigenline bundle, whose quotient action has multiplier $cz+d$ by \eqref{eq:automorphy}. On the other hand,
\[
 \gamma'(z)=(cz+d)^{-2}.
\]
Thus the quotient action on $T\Sigma$ has multiplier $(cz+d)^{-2}$, whereas the action on $K_\Sigma$ has multiplier $(cz+d)^2$. Squaring the multiplier of $(F,I)$ proves \eqref{eq:half-canonical} as an isomorphism of complex line bundles, and hence yields the stated integral first Chern class identity.
\end{proof}

\begin{remark}[A curvature check of the sign]
The eigenline in the proof is a holomorphic subbundle of the flat complex rank-two bundle. In the holomorphic frame $s(z)$, the Hermitian metric induced by \eqref{eq:compatible-metric} has squared norm $h(z)=v$, up to a positive constant. Indeed, $s(z)=w-\ii I_z w$ for $w=(u,1)$, and $\varepsilon(w,I_z w)=v$. The metric descends because
\[
\operatorname{Im}(\gamma z)=v/|cz+d|^2,
\]
in agreement with \eqref{eq:automorphy}. Its first Chern form is
\[
-\frac{\ii}{2\pi}\partial\bar\partial\log v=\frac{du\wedge dv}{4\pi v^2}.
\]
Since the hyperbolic area of $\Sigma$ is $4\pi(g-1)$, integration again gives $\deg(F,I)=g-1$. This independently confirms the sign fixed by $\varepsilon$-compatibility.
\end{remark}

Let $E_r$ be the flat real symplectic bundle associated with \eqref{eq:torus-data}:
\[
 E_r=(\HH\times V_r)/\Gamma\longrightarrow\Sigma.
\]
Equip each of the $2r$ column planes in \cref{lem:form} with $I_z$. Their direct sum defines a $\Gamma$-equivariant complex structure $I^{(r)}$ compatible with $\Omega_r$. As complex vector bundles,
\begin{equation}\label{eq:vertical-chern}
(E_r,I^{(r)})\simeq(F,I)^{\oplus2r},\qquad c_1(E_r,I^{(r)})=r\,c_1(K_\Sigma).
\end{equation}
This decomposition takes place at the level of real representations and their associated complex vector bundles; it need not preserve the integral lattice, and the construction does not require it to do so.

\section{Flat torus bundles and the counterexample}\label{sec:bundle}

\subsection{The quotient and its symplectic forms}

For $r\ge1$, define
\begin{equation}\label{eq:Mr}
M_r=(\HH\times T_r)/\Gamma,\qquad\gamma\cdot(z,[v])=(\gamma z,[\gamma v]).
\end{equation}
Projection to the first factor defines a smooth torus bundle $q_r:M_r\to\Sigma$ with fiber $T_r\simeq T^{4r}$. Since the action of $\Gamma$ on $\HH$ is free and properly discontinuous, so is its diagonal action on $\HH\times T_r$. Thus $M_r$ is a smooth connected manifold. Its base and fiber are compact, so $M_r$ is closed. The fiberwise origins are fixed by the linear monodromy and define a section of $q_r$.

The constant form $\Omega_r$ on $V_r$ is invariant under both lattice translations and the action of $\Gamma$. It therefore descends to a closed two-form $\eta_r$ on $M_r$, vanishing on the flat horizontal distribution. Fix a positive area form $\sigma$ on $\Sigma$ and, for $t>0$, define
\begin{equation}\label{eq:omega}
 \omega_{r,t}=\eta_r+t\,q_r^*\sigma.
\end{equation}
The horizontal and vertical distributions are preserved by the quotient action. Relative to their direct sum, $\omega_{r,t}$ has blocks $t\sigma$ and $\Omega_r$, with zero mixed terms. It is therefore symplectic for every $t>0$.

Equivalently, the group $\Lambda_r\rtimes\Gamma$ acts on $\HH\times V_r$ by
\[
 (\lambda,\gamma)\cdot(z,v)=(\gamma z,\lambda+\gamma v),
\]
and its quotient is $M_r$. The action is free: a fixed point first forces $\gamma=1$, by freeness on $\HH$, and then $\lambda=0$. It is properly discontinuous because a compact subset of $\HH$ allows only finitely many $\gamma$, and for each of these, compact subsets of $V_r$ allow only finitely many lattice translations $\lambda$. Since $\HH\times V_r$ is contractible, this is the universal covering space. In particular,
\begin{equation}\label{eq:aspherical}
\pi_1(M_r)\simeq\Lambda_r\rtimes\Gamma,\qquad \pi_k(M_r)=0\quad(k\ge2).
\end{equation}

\subsection{The tangent bundle and first Chern degree}

Torus translations act trivially on tangent spaces, while the derivative of the linear monodromy is the representation on $V_r$. Thus the vertical tangent bundle is $q_r^*E_r$. The flat horizontal distribution yields a real vector-bundle splitting
\[
 TM_r\simeq q_r^*T\Sigma\oplus q_r^*E_r.
\]
Equip the horizontal summand with the complex structure $j_\Sigma$ of $\Sigma$ and the vertical summand with $I^{(r)}$. The resulting almost complex structure $J_r$ is compatible with $\omega_{r,t}$ for every $t>0$.

\begin{proposition}\label{prop:family}
For every integer $r\ge1$ and every $t>0$,
\begin{equation}\label{eq:chern-family}
c_1(M_r,\omega_{r,t})=(r-1)q_r^*c_1(K_\Sigma)\quad\text{in }H^2(M_r;\Z).
\end{equation}
Moreover,
\begin{equation}\label{eq:degree-family}
\bigl\langle c_1(M_r,\omega_{r,t})\smile[\omega_{r,t}]^{2r},[M_r]\bigr\rangle=(r-1)(2g-2)(2r)!\,3^r.
\end{equation}
\end{proposition}

\begin{proof}
By the complex tangent-bundle splitting and \eqref{eq:vertical-chern},
\[
\begin{aligned}
c_1(TM_r,J_r)=q_r^*\bigl(c_1(T\Sigma)+c_1(E_r,I^{(r)})\bigr)=(r-1)q_r^*c_1(K_\Sigma),
\end{aligned}
\]
which proves \eqref{eq:chern-family}. Indeed, at the level of smooth complex line bundles,
\begin{equation}\label{eq:determinant}
\det_{\C}(TM_r,J_r)\simeq q_r^*\bigl(K_\Sigma^{-1}\otimes F^{\otimes2r}\bigr)\simeq q_r^*K_\Sigma^{\otimes(r-1)}.
\end{equation}

Choose a closed real two-form $\theta$ representing $c_1(K_\Sigma)$, so that $\int_\Sigma\theta=2g-2$. Since $\theta\wedge\sigma=0$ on $\Sigma$, fiber integration gives
\[
\begin{aligned}
\int_{M_r}(r-1)q_r^*\theta\wedge\omega_{r,t}^{2r}=(r-1)\int_{M_r}q_r^*\theta\wedge\eta_r^{2r}=(r-1)(2g-2)\int_{T_r}\Omega_r^{2r}.
\end{aligned}
\]
All orientations are symplectic; the orientation of $M_r$ agrees with the product orientation of its base and fiber.

In the lattice coordinates $(x_0,\ldots,x_{4r-1})$ from the ordered quaternionic bases, \eqref{eq:Q} reads
\[
\Omega_r=\sum_{\ell=0}^{r-1}\bigl(dx_{4\ell}\wedge dx_{4\ell+1}+3\,dx_{4\ell+2}\wedge dx_{4\ell+3}\bigr).
\]
The fundamental lattice cube has coordinate volume $1$, so
\[
 \int_{T_r}\Omega_r^{2r}=(2r)!\,3^r.
\]
This proves \eqref{eq:degree-family}.
\end{proof}

\subsection{The absence of pseudoholomorphic spheres}

For completeness, we record the area argument, which requires no genericity assumption on the almost complex structure.

\begin{lemma}\label{lem:no-spheres}
Let $(X,\omega_X)$ be a symplectic manifold such that $\pi_2(X)=0$. If $J'$ is $\omega_X$-tame, every $J'$-holomorphic map $u:\mathbb P^1\to X$ is constant.
\end{lemma}

\begin{proof}
Since $\pi_2(X)=0$, the map $u$ is null-homotopic. As $\omega_X$ is closed, Stokes' theorem gives
\[
 \int_{\mathbb P^1}u^*\omega_X=0.
\]
In a positively oriented holomorphic coordinate $s+\ii t$ on the domain, the pseudoholomorphic equation yields
\[
u^*\omega_X(\partial_s,\partial_t)=\omega_X\bigl(du(\partial_s),J'du(\partial_s)\bigr)\ge0.
\]
If $u$ is nonconstant, then $du$ is nonzero somewhere. By tameness and continuity, the displayed expression is strictly positive on a neighborhood of that point, contradicting the vanishing of the total area.
\end{proof}

\begin{proof}[Proof of \cref{thm:main}]
Fix the surface $\Sigma=\Gamma\backslash\HH$ of \cref{sec:arithmetic} and choose $t>0$. For $r=2$, \cref{prop:family} gives
\[
c_1(M_2,\omega_{2,t})=q_2^*c_1(K_\Sigma),\qquad
\bigl\langle c_1(M_2,\omega_{2,t})\smile[\omega_{2,t}]^4,[M_2]\bigr\rangle=216(2g-2).
\]
Let $n\ge5$ and set $k=n-5$. On the unit-period torus $T^{2k}=(\R/\Z)^{2k}$, take the standard symplectic form
\[
\tau_k=\sum_{j=1}^k dx_j\wedge dy_j,\qquad \int_{T^{2k}}\tau_k^k=k!.
\]
For $k=0$, the torus is a point and the integral is understood to be $1$. Set
\begin{equation}\label{eq:general-construction}
\begin{aligned}
X_n=M_2\times T^{2k},\quad\omega_n=\operatorname{pr}_{M_2}^*\omega_{2,t}+\operatorname{pr}_T^*\tau_k\qquad\textrm{and}\qquad p_n=q_2\circ\operatorname{pr}_{M_2}.
\end{aligned}
\end{equation}
Then $X_n$ is a closed connected symplectic manifold of real dimension $2n$, and $p_n$ is a smooth $T^{2n-2}$-bundle over $\Sigma$. By \eqref{eq:aspherical}, its universal cover is
\[
\widetilde X_n\cong\HH\times\R^8\times\R^{2k}=\HH\times\R^{2n-2}.
\]
Equip $X_n$ with the product compatible almost complex structure $\widehat J_n=J_2\oplus J_{\mathrm{std}}$. The standard complex tangent bundle of $T^{2k}$ is trivial, so
\[
c_1(X_n,\omega_n)=\operatorname{pr}_{M_2}^*c_1(M_2,\omega_{2,t})=p_n^*c_1(K_\Sigma).
\]
Only the term involving four copies of $\omega_{2,t}$ and $k$ copies of $\tau_k$ contributes to the first Chern degree. Its binomial coefficient is $\binom{n-1}{4}$, so
\begin{align*}
\bigl\langle c_1(X_n,\omega_n)\smile[\omega_n]^{n-1},[X_n]\bigr\rangle=\binom{n-1}{4}\,216(2g-2)\,k!=9(n-1)!(2g-2)>0.
\end{align*}
Finally, $\pi_2(X_n)=0$, so \cref{lem:no-spheres} rules out nonconstant pseudoholomorphic spheres for every $\omega_n$-tame almost complex structure.
\end{proof}

\begin{remark}[A uniform quotient presentation]
The product construction \eqref{eq:general-construction} also has a direct presentation as a flat torus bundle. Set
\[
\widehat V_n=B_\R^{\oplus2}\oplus\R^{2n-10},\qquad\widehat\Lambda_n=\cO^{\oplus2}\oplus\Z^{2n-10},
\]
and let $\Gamma$ act by left multiplication on the quaternionic summands and trivially on the last summand. Then
\[
 X_n=\bigl(\HH\times(\widehat V_n/\widehat\Lambda_n)\bigr)/\Gamma.
\]
The invariant fiber form is $\Omega_B\oplus\Omega_B\oplus\tau_{n-5}$, with $\tau_{n-5}$ now denoting the corresponding constant form on $\R^{2n-10}$. Adding $t\,p_n^*\sigma$ to its descended form recovers $\omega_n$. The compatible complex tangent bundle is
\begin{equation}\label{eq:uniform-tangent}
(TX_n,\widehat J_n)\simeq p_n^*\bigl(T\Sigma\oplus F^{\oplus4}\oplus\underline{\C}^{\,n-5}\bigr).
\end{equation}
In particular, its complex determinant line bundle is $p_n^*K_\Sigma$ for every $n\ge5$.
\end{remark}

\section{Further consequences and scope}\label{sec:consequences}

\subsection{Nonintegrability of the displayed almost complex structures}

The compatible almost complex structures $J_r$ on $M_r$ and $\widehat J_n$ on $X_n$ are not integrable. Both assertions follow from the same local computation on a quaternionic column plane.

Choose real coordinates $(s,t)$ on one standard column plane. By \eqref{eq:Iz}, its vertical $(0,1)$-line is spanned by
\[
 W=\bar z\,\partial_s+\partial_t.
\]
The horizontal vector $\partial_{\bar z}$ is also of type $(0,1)$, whereas
\[
 [\partial_{\bar z},W]=\partial_s.
\]
The vector $\partial_s$ is not a complex multiple of $W$ and has no components in the other summands. Hence it does not lie in the $(0,1)$-subspace, so the $(0,1)$-distribution is not involutive for either $J_r$ or $\widehat J_n$. This shows nonintegrability only for the specified structures; it makes no claim about other complex structures on the underlying smooth manifolds.

\subsection{Numerical positivity and dimensional scope}

The positivity in \cref{thm:main} is precisely the numerical condition \eqref{eq:bdpp-condition}. By \eqref{eq:uniform-tangent}, the complex tangent bundle of $X_n$ is pulled back from a bundle over a Riemann surface. Consequently, in integral cohomology,
\[
c_j(TX_n,\widehat J_n)=0\quad(j\ge2),\qquad c_1(TX_n,\widehat J_n)^2=0.
\]
For every $n\ge5$, the class $c_1(X_n,\omega_n)$ is nonzero: the fiberwise origins define a section of $p_n$, so pullback from $\Sigma$ is injective. Nevertheless, its square vanishes, and therefore it cannot be a positive multiple of $[\omega_n]$. Moreover, $c_1(X_n,\omega_n)^n=0$, so this class cannot be represented by a symplectic form.

For $r=1$, \eqref{eq:determinant} gives a trivial complex determinant line bundle and hence
\[
 c_1(M_1,\omega_{1,t})=0.
\]
The manifold $M_r$ is the $r$-fold fiber product of $M_1\to\Sigma$ over the same base. Its vertical Chern contribution grows with $r$, while the contribution from $T\Sigma$ occurs only once. Thus the first positive-degree member of this family is $M_2$, in real dimension ten. Products with standard symplectic tori yield \cref{thm:main} for every $n\ge5$. This construction does not produce positive-degree examples in real dimensions six or eight and does not resolve the corresponding cases of BDPP's question.

\Needspace{12\baselineskip}
\section*{Acknowledgments}

The author used ChatGPT to assist with the exploration of the construction and the preparation of the manuscript. The author verified and completed all mathematical arguments and takes full responsibility for the content of this paper.


\begin{thebibliography}{99}
\raggedright
\small

\bibitem{BDPP13}
S.~Boucksom, J.-P.~Demailly, M.~P\u{a}un, and T.~Peternell,
\emph{The pseudo-effective cone of a compact K\"ahler manifold and varieties of negative Kodaira dimension},
J. Algebraic Geom. \textbf{22} (2013), no.~2, 201--248.
\href{https://doi.org/10.1090/S1056-3911-2012-00574-8}{doi:10.1090/S1056-3911-2012-00574-8}.

\bibitem{FP10}
J.~Fine and D.~Panov,
\emph{Hyperbolic geometry and non-K\"ahler manifolds with trivial canonical bundle},
Geom. Topol. \textbf{14} (2010), no.~3, 1723--1763.
\href{https://doi.org/10.2140/gt.2010.14.1723}{doi:10.2140/gt.2010.14.1723}.

\bibitem{MS12}
D.~McDuff and D.~Salamon,
\emph{$J$-holomorphic Curves and Symplectic Topology},
second edition, American Mathematical Society Colloquium Publications, vol.~52,
American Mathematical Society, Providence, RI, 2012.

\bibitem{Milnor58}
J.~Milnor,
\emph{On the existence of a connection with curvature zero},
Comment. Math. Helv. \textbf{32} (1958), 215--223.
\href{https://doi.org/10.1007/BF02564579}{doi:10.1007/BF02564579}.

\bibitem{VZ04}
E.~Viehweg and K.~Zuo,
\emph{A characterization of certain Shimura curves in the moduli stack of abelian varieties},
J. Differential Geom. \textbf{66} (2004), no.~2, 233--287.
\href{https://doi.org/10.4310/jdg/1102538611}{doi:10.4310/jdg/1102538611}.
\href{https://arxiv.org/abs/math/0207228}{arXiv:math/0207228}.

\end{thebibliography}
\end{document}